\documentclass[12pt, reqno]{amsart}
\usepackage{amsmath, amsthm, amscd, amsfonts, amssymb, graphicx, color}
\usepackage[bookmarksnumbered, colorlinks, plainpages]{hyperref}
\hypersetup{colorlinks=true,linkcolor=red, anchorcolor=green, citecolor=cyan, urlcolor=red, filecolor=magenta, pdftoolbar=true}

\newtheorem{theorem}{Theorem}[section]
\newtheorem{lemma}[theorem]{Lemma}

\newtheorem{corollary}[theorem]{Corollary}
\theoremstyle{definition}
\newtheorem{definition}[theorem]{Definition}
\newtheorem{example}[theorem]{Example}

\theoremstyle{remark}

\newtheorem{remark}[theorem]{Remark}
\numberwithin{equation}{section}

\begin{document}

\setcounter{page}{1}

\title[$\ast$-operator frame for $End_{\mathcal{A}}^{\ast}(\mathcal{H})$]{$\ast$-operator frame for $End_{\mathcal{A}}^{\ast}(\mathcal{H})$}

\author[S. KABBAJ, M. ROSSAFI]{S. KABBAJ$^1$ \MakeLowercase{and} M. ROSSAFI$^2$$^{*}$}

\address{$^{1}$Department of Mathematics, University of Ibn Tofail, B.P. 133, Kenitra, Morocco}
\email{\textcolor[rgb]{0.00,0.00,0.84}{samkabbaj@yahoo.fr}}

\address{$^{2}$Department of Mathematics, University of Ibn Tofail, B.P. 133, Kenitra, Morocco}
\email{\textcolor[rgb]{0.00,0.00,0.84}{rossafimohamed@gmail.com}}



\date{12/12/2017; 
\newline \indent $^{*}$Corresponding author}

\begin{abstract}
In this paper, a new notion of frames is introduced: $\ast$-operator frame as generalization of $\ast$-frames in Hilbert $C^{\ast}$-modules introduced by A. Alijani and M. A. Dehghan \cite{Ali} and we establish some results.
\end{abstract} \maketitle

\section{Introduction}
Frames were first introduced in 1952 by Duffin and Schaeffer \cite{Duf} in the study of nonharmonic fourier series. Frames possess many nice properties which make them very useful in wavelet analysis, irregular sampling theory, signal processing and many other fields. The frame theory was quickly generalized.
There were several kinds of generalization. For example: g-frames, $\ast$-frames, $\ast$-g-frames and $\ast$-$K$-g-frames in Hilbert $A$-modules \cite{Ros2}.

In this article, a new notion of frames is introduced: $\ast$-operator frame as generalization of $\ast$-frames in Hilbert $C^{\ast}$-modules introduced by A. Alijani and M. A. Dehghan \cite{Ali} and we establish some results.

The paper is organized as follows:

 in section 2, we briefly recall the definitions and basic properties of $C^{\ast}$-algebra, Hilbert $C^{\ast}$-modules, frames and $\ast$-frames in Hilbert $C^{\ast}$-modules; 
 
 in section 3, we introduce the $\ast$-operator frame, the $\ast$-operator frame transform and the $\ast$-frame operator;
 
 in section 4, we investigate tensor product of Hilbert $C^{\ast}$-modules, we show that tensor product of $\ast$-operator frames for Hilbert $C^{\ast}$-modules $\mathcal{H}$ and $\mathcal{K}$, present an $\ast$-operator frames for $\mathcal{H}\otimes\mathcal{K}$, and tensor product of their $\ast$-frame operators is the $\ast$-frame operator of their tensor product of $\ast$-operator frames.
\section{Preliminaries}
Let $I$ be a countable index set. In this section we briefly recall the definitions and basic properties of $C^{\ast}$-algebra, Hilbert $C^{\ast}$-modules, frame, $\ast$-frame in Hilbert $C^{\ast}$-modules. For information about frames in Hilbert spaces we refer to \cite{Ch}. Our reference for $C^{\ast}$-algebras is \cite{Dav,Con}. For a $C^{\ast}$-algebra $\mathcal{A}$, an element $a\in\mathcal{A}$ is positive ($a\geq 0$) if $a=a^{\ast}$ and $sp(a)\subset\mathbf{R^{+}}$. $\mathcal{A}^{+}$ denotes the set of positive elements of $\mathcal{A}$.
\begin{definition}
	\cite{BA}. Let $ \mathcal{A} $ be a unital $C^{\ast}$-algebra and $\mathcal{H}$ be a left $ \mathcal{A} $-module, such that the linear structures of $\mathcal{A}$ and $ \mathcal{H} $ are compatible. $\mathcal{H}$ is a pre-Hilbert $\mathcal{A}$-module if $\mathcal{H}$ is equipped with an $\mathcal{A}$-valued inner product $\langle.,.\rangle_{\mathcal{A}} :\mathcal{H}\times\mathcal{H}\rightarrow\mathcal{A}$, such that is sesquilinear, positive definite and respects the module action. In the other words,
	\begin{itemize}
		\item [(i)] $ \langle x,x\rangle_{\mathcal{A}}\geq0 $ for all $ x\in\mathcal{H} $ and $ \langle x,x\rangle_{\mathcal{A}}=0$ if and only if $x=0$.
		\item [(ii)] $\langle ax+y,z\rangle_{\mathcal{A}}=a\langle x,y\rangle_{\mathcal{A}}+\langle y,z\rangle_{\mathcal{A}}$ for all $a\in\mathcal{A}$ and $x,y,z\in\mathcal{H}$.
		\item[(iii)] $ \langle x,y\rangle_{\mathcal{A}}=\langle y,x\rangle_{\mathcal{A}}^{\ast} $ for all $x,y\in\mathcal{H}$.
	\end{itemize}	 
	For $x\in\mathcal{H}, $ we define $||x||=||\langle x,x\rangle_{\mathcal{A}}||^{\frac{1}{2}}$. If $\mathcal{H}$ is complete with $||.||$, it is called a Hilbert $\mathcal{A}$-module or a Hilbert $C^{\ast}$-module over $\mathcal{A}$. For every $a$ in $C^{\ast}$-algebra $\mathcal{A}$, we have $|a|=(a^{\ast}a)^{\frac{1}{2}}$ and the $\mathcal{A}$-valued norm on $\mathcal{H}$ is defined by $|x|=\langle x, x\rangle_{\mathcal{A}}^{\frac{1}{2}}$ for $x\in\mathcal{H}$.
	\begin{example} \cite{Tro}
		If $ \{\mathcal{H}_{k}\}_{k\in\mathbf{N}} $ is a countable set of Hilbert $\mathcal{A}$-modules, then one can define their direct sum $ \oplus_{k\in\mathbb{N}}\mathcal{H}_{k} $. On the $\mathcal{A}$-module $ \oplus_{k\in\mathbb{N}}\mathcal{H}_{k} $ of all sequences $x=(x_{k})_{k\in\mathbb{N}}: x_{k}\in\mathcal{H}_{k}$, such that the series $ \sum_{k\in\mathbb{N}}\langle x_{k}, x_{k}\rangle_{\mathcal{A}} $ is norm-convergent in the $\mathcal{C}^{\ast}$-algebra $\mathcal{A}$, we define the inner product by
		\begin{equation*}
		\langle x, y\rangle:=\sum_{k\in\mathbb{N}}\langle x_{k}, y_{k}\rangle_{\mathcal{A}} 
		\end{equation*}
		for $x, y\in\oplus_{k\in\mathbb{N}}\mathcal{H}_{k} $.
		
		Then $\oplus_{k\in\mathbb{N}}\mathcal{H}_{k}$ is a Hilbert $\mathcal{A}$-module.
		
		The direct sum of a countable number of copies of a Hilbert $\mathcal{C}^{\ast}$-module $\mathcal{H}$ is denoted by $l^{2}(\mathcal{H})$.
	\end{example}
\end{definition}
	Let $\mathcal{H}$ and $\mathcal{K}$ be two Hilbert $\mathcal{A}$-modules. A map $T:\mathcal{H}\rightarrow\mathcal{K}$ is said to be adjointable if there exists a map $T^{\ast}:\mathcal{K}\rightarrow\mathcal{H}$ such that $\langle Tx,y\rangle_{\mathcal{A}}=\langle x,T^{\ast}y\rangle_{\mathcal{A}}$ for all $x\in\mathcal{H}$ and $y\in\mathcal{K}$.
	 
	We also reserve the notation $End_{\mathcal{A}}^{\ast}(\mathcal{H},\mathcal{K})$ for the set of all adjointable operators from $\mathcal{H}$ to $\mathcal{K}$ and $End_{\mathcal{A}}^{\ast}(\mathcal{H},\mathcal{H})$ is abbreviated to $End_{\mathcal{A}}^{\ast}(\mathcal{H})$.

\begin{definition} 
	\cite{F4}. Let $ \mathcal{H} $ be a Hilbert $\mathcal{A}$-module. A family $\{x_{i}\}_{i\in I}$ of elements of $\mathcal{H}$ is a frame for $ \mathcal{H} $, if there exist two positive constants $A$ and $B$ such that for all $x\in\mathcal{H}$,
	\begin{equation}\label{1}
		A\langle x,x\rangle_{\mathcal{A}}\leq\sum_{i\in I}\langle x,x_{i}\rangle_{\mathcal{A}}\langle x_{i},x\rangle_{\mathcal{A}}\leq B\langle x,x\rangle_{\mathcal{A}}.
	\end{equation}
	The numbers $A$ and $B$ are called lower and upper bound of the frame, respectively. If $A=B=\lambda$, the frame is $\lambda$-tight. If $A = B = 1$, it is called a normalized tight frame or a Parseval frame. If the sum in the middle of \eqref{1} is convergent in norm, the frame is called standard.
\end{definition}	
\begin{definition}
	\cite{Ali}. Let $ \mathcal{H} $ be a Hilbert $\mathcal{A}$-module. A family $\{x_{i}\}_{i\in I}$ of elements of $\mathcal{H}$ is an $\ast$-frame for $ \mathcal{H} $, if there	exist strictly nonzero elements $A$ , $B$ in $\mathcal{A}$, such that for all $x\in\mathcal{H}$,
	\begin{equation}\label{3}
		A\langle x,x\rangle_{\mathcal{A}} A^{\ast}\leq\sum_{i\in I}\langle x,x_{i}\rangle_{\mathcal{A}}\langle x_{i},x\rangle_{\mathcal{A}}\leq B\langle x,x\rangle_{\mathcal{A}} B^{\ast}.
	\end{equation}
	The elements $A$ and $B$ are called lower and upper bound of the $\ast$-frame, respectively. If $A=B=\lambda$, the $\ast$-frame is $\lambda$-tight. If $A = B = 1$, it is called a normalized tight $\ast$-frame or a Parseval $\ast$-frame. If the sum in the middle of \eqref{3} is convergent in norm, the $\ast$-frame is called standard.
\end{definition}
\begin{lemma}\label{2.7}
	\cite{Ali}. If $\varphi:\mathcal{A}\longrightarrow\mathcal{B}$ is an $\ast$-homomorphism between $\mathcal{C}^{\ast}$-algebras, then $\varphi$ is increasing, that is, if $a\leq b$, then $\varphi(a)\leq\varphi(b)$.
\end{lemma}
\begin{lemma}\label{2.8}
	\cite{Ali}. Let $\mathcal{H}$ and $\mathcal{K}$ be two Hilbert $\mathcal{A}$-modules and $T\in End_{\mathcal{A}}^{\ast}(\mathcal{H},\mathcal{K})$.
	\begin{itemize}
		\item [(i)] If $T$ is injective and $T$ has closed range, then the adjointable map $T^{\ast}T$ is invertible and $$\|(T^{\ast}T)^{-1}\|^{-1}I_{\mathcal{H}}\leq T^{\ast}T\leq\|T\|^{2}I_{\mathcal{H}}.$$
		\item  [(ii)]	If $T$ is surjective, then the adjointable map $TT^{\ast}$ is invertible and $$\|(TT^{\ast})^{-1}\|^{-1}I_{\mathcal{K}}\leq TT^{\ast}\leq\|T\|^{2}I_{\mathcal{K}}.$$
\end{itemize}	
\end{lemma}
\section{$\ast$-operator frame for $End_{\mathcal{A}}^{\ast}(\mathcal{H})$}
We begin this section with the following definition.
\begin{definition}
	A family of adjointable operators $\{T_{i}\}_{i\in I}$ on a Hilbert $\mathcal{A}$-module $\mathcal{H}$ over a unital $C^{\ast}$-algebra is said to be an operator frame for $End_{\mathcal{A}}^{\ast}(\mathcal{H})$, if there exist positive constants $A, B > 0$ such that 
	\begin{equation}\label{eq3}
		A\langle x,x\rangle\leq\sum_{i\in I}\langle T_{i}x,T_{i}x\rangle\leq B\langle x,x\rangle, \forall x\in\mathcal{H}.
	\end{equation}
	The numbers $A$ and $B$ are called lower and upper bound of the operator frame, respectively. If $A=B=\lambda$, the operator frame is $\lambda$-tight. If $A = B = 1$, it is called a normalized tight operator frame or a Parseval operator frame. If only upper inequality of \eqref{eq3} hold, then $\{T_{i}\}_{i\in I}$ is called an operator Bessel sequence for $End_{\mathcal{A}}^{\ast}(\mathcal{H})$.
	
	If the sum in the middle of \eqref{eq3} is convergent in norm, the operator frame is called standard.
\end{definition}
Throughout the paper, series like \eqref{eq3} are assumed to be convergent in the norm sense.
\begin{example}
	Let $\mathcal{A}$ be a Hilbert $C^{\ast}$-module over itself with the inner product $\langle a,b\rangle=ab^{\ast}$.
		Let $\{x_{i}\}_{i\in I}$ be a frame for $\mathcal{A}$ with bounds $A$ and $B$. For each $i\in I$, we define $T_{i}:\mathcal{A}\to\mathcal{A}$ by $T_{i}x=\langle x,x_{i}\rangle,\;\; \forall x\in\mathcal{A}$. $T_{i}$ is adjointable and $T_{i}^{\ast}a=ax_{i}$ for each $a\in\mathcal{A}$. And we have 
	\begin{equation*}
	A\langle x,x\rangle\leq\sum_{i\in I}\langle x,x_{i}\rangle\langle x_{i},x\rangle\leq B\langle x,x\rangle, \forall x\in\mathcal{A}.
	\end{equation*}
	Then
	\begin{equation*}
	A\langle x,x\rangle\leq\sum_{i\in I}\langle T_{i}x,T_{i}x\rangle\leq B\langle x,x\rangle, \forall x\in\mathcal{A}.
	\end{equation*}
	So $\{T_{i}\}_{i\in I}$ is an operator frame in $\mathcal{A}$ with bounds $A$ and $B$.
\end{example}
\begin{example}
	Let $\mathcal{H}$ and $\mathcal{K}$ be separable Hilbert spaces and let $B(\mathcal{H,K)}$ be the set of all bounded linear operators from $\mathcal{H}$ into $\mathcal{K}$. $B(\mathcal{H,K)}$ is a Hilbert $B(\mathcal{K})$-module with a $B(\mathcal{K})$-valued inner product $\langle S, T\rangle= ST^{\ast}$ for all $S, T\in B(\mathcal{H,K)}$, and with a linear operation of $B(\mathcal{K})$ on $B(\mathcal{H,K)}$ by composition of operators.
	
	Let $I=\mathbf{N}$ and fix $(a_{i})_{i\in\mathbf{N}}\in l^{2}(\mathbf{C})$. 
	
	Define: $T_{i}(X)=a_{i}X, \forall X\in B(\mathcal{H,K)}, \forall i\in\mathbb{N}$. We have 
	\begin{equation*}
	\sum_{i\in\mathbb{N}}\langle T_{i}x,T_{i}x\rangle=\sum_{i\in\mathbb{N}}|a_{i}|^{2}\langle X, X\rangle, \forall X\in B(\mathcal{H,K)}.
	\end{equation*}
	Then $\{T_{i}\}_{i\in\mathbb{N}}$ is $\sum_{i\in\mathbb{N}}|a_{i}|^{2}$-tight operator frame.
\end{example}
Now we define the $\ast$-operator frame for $End_{\mathcal{A}}^{\ast}(\mathcal{H})$.
\begin{definition}
	A family of adjointable operators $\{T_{i}\}_{i\in I}$ on a Hilbert $\mathcal{A}$-module $\mathcal{H}$ over a unital $C^{\ast}$-algebra is said to be an $\ast$-operator frame for $End_{\mathcal{A}}^{\ast}(\mathcal{H})$, if there exists two strictly nonzero elements $A$ and $B$ in $\mathcal{A}$ such that 
	\begin{equation}\label{eqq33}
		A\langle x,x\rangle A^{\ast}\leq\sum_{i\in I}\langle T_{i}x,T_{i}x\rangle\leq B\langle x,x\rangle B^{\ast}, \forall x\in\mathcal{H}.
	\end{equation}
	The elements $A$ and $B$ are called lower and upper bounds of the $\ast$-operator frame, respectively. If $A=B=\lambda$, the $\ast$-operator frame is $\lambda$-tight. If $A = B = 1_{\mathcal{A}}$, it is called a normalized tight $\ast$-operator frame or a Parseval $\ast$-operator frame. If only upper inequality of \eqref{eqq33} hold, then $\{T_{i}\}_{i\in i}$ is called an $\ast$-operator Bessel sequence for $End_{\mathcal{A}}^{\ast}(\mathcal{H})$.
\end{definition}
We mentioned that the set of all of operator frames for $End_{\mathcal{A}}^{\ast}(\mathcal{H})$ can be considered
as a subset of $\ast$-operator frame. To illustrate this, let $\{T_{j}\}_{i\in I}$ be an operator frame for Hilbert $\mathcal{A}$-module $\mathcal{H}$
with operator frame real bounds $A$ and $B$. Note that for $x\in\mathcal{H}$,
\begin{equation*}
(\sqrt{A})1_{\mathcal{A}}\langle x,x\rangle_{\mathcal{A}}(\sqrt{A})1_{\mathcal{A}}\leq\sum_{i\in I}\langle T_{i}x,T_{i}x\rangle\leq(\sqrt{B})1_{\mathcal{A}}\langle x,x\rangle_{\mathcal{A}}(\sqrt{B})1_{\mathcal{A}}.
\end{equation*}
Therefore, every operator frame for $End_{\mathcal{A}}^{\ast}(\mathcal{H})$ with real bounds $A$ and $B$ is an $\ast$-operator frame for $End_{\mathcal{A}}^{\ast}(\mathcal{H})$ with $\mathcal{A}$-valued $\ast$-operator frame bounds $(\sqrt{A})1_{\mathcal{A}}$ and $(\sqrt{B})1_{\mathcal{B}}$.
\begin{example}
	Let $\mathcal{A}$ be a Hilbert $C^{\ast}$-module over itself with the inner product $\langle a,b\rangle=ab^{\ast}$.
	Let $\{x_{i}\}_{i\in I}$ be an $\ast$-frame for $\mathcal{A}$ with bounds $A$ and $B$, respectively. For each $i\in I$, we define $T_{i}:\mathcal{A}\to\mathcal{A}$ by $T_{i}x=\langle x,x_{i}\rangle,\;\; \forall x\in\mathcal{A}$. $T_{i}$ is adjointable and $T_{i}^{\ast}a=ax_{i}$ for each $a\in\mathcal{A}$. And we have 
	\begin{equation*}
	A\langle x,x\rangle A^{\ast}\leq\sum_{i\in I}\langle x,x_{i}\rangle\langle x_{i},x\rangle\leq B\langle x,x\rangle B^{\ast}, \forall x\in\mathcal{A}.
	\end{equation*}
	Then
	\begin{equation*}
	A\langle x,x\rangle A^{\ast}\leq\sum_{i\in I}\langle T_{i}x,T_{i}x\rangle\leq B\langle x,x\rangle B^{\ast}, \forall x\in\mathcal{A}.
	\end{equation*}
	So $\{T_{i}\}_{i\in I}$ is an $\ast$-operator frame in $\mathcal{A}$ with bounds $A$ and $B$, respectively.
\end{example}
\begin{remark}
The examples $3.3$ and $3.4$ in \cite{A} are examples of $\ast$-operator frame.
\end{remark}
Similar to $\ast$-frames, we introduce the $\ast$-operator frame transform and $\ast$-frame operator and establish some properties.
\begin{theorem} \label{2.3}
	Let $\{T_{i}\}_{i\in I}\subset End_{\mathcal{A}}^{\ast}(\mathcal{H})$ be an $\ast$-operator frame with lower and upper bounds $A$ and $B$, respectively. The $\ast$-operator frame transform $R:\mathcal{H}\rightarrow l^{2}(\mathcal{H})$ defined by $Rx=\{T_{i}x\}_{i\in I}$ is injective and closed range adjointable $\mathcal{A}$-module map and $\|R\|\leq\|B\|$. The adjoint operator $R^{\ast}$ is surjective and it is given by $R^{\ast}(\{x_{i}\}_{i\in I})=\sum_{i\in I}T_{i}^{\ast}x_{i}$ for all $\{x_{i}\}_{i\in I}$ in $l^{2}(\mathcal{H})$.
\end{theorem}
\begin{proof}
	By the definition of norm in $l^{2}(\mathcal{H})$
	\begin{equation}\label{3.3}
	\|Rx\|^{2}=\|\sum_{i\in I}\langle T_{i}x,T_{i}x\rangle\|\leq\|B\|^{2}\|\langle x,x\rangle\|, \forall x\in\mathcal{H}.
	\end{equation}
	This inequality implies that $R$ is well defined and $\|R\|\leq\|B\|$. Clearly, $R$ is a linear $\mathcal{A}$-module map. We now show that the range of $R$ is closed. Let $\{Rx_{n}\}_{n\in\mathbb{N}}$ be a sequence in the range of $R$ such that $\lim_{n\to\infty}Rx_{n}=y$. For $n, m\in\mathbb{N}$, we have
	\begin{equation*}
	\|A\langle x_{n}-x_{m},x_{n}-x_{m}\rangle A^{\ast}\|\leq\|\langle R(x_{n}-x_{m}),R(x_{n}-x_{m})\rangle\|=\|R(x_{n}-x_{m})\|^{2}.
	\end{equation*}
		Since $\{Rx_{n}\}_{n\in\mathbb{N}}$ is Cauchy sequence in $\mathcal{H}$, then
	
	$\|A\langle x_{n}-x_{m},x_{n}-x_{m}\rangle A^{\ast}\|\rightarrow0$, as $n,m\rightarrow\infty.$
	
	Note that for $n, m\in\mathbb{N}$,
	\begin{align*}
\|\langle x_{n}-x_{m},x_{n}-x_{m}\rangle\|&=\|A^{-1}A\langle x_{n}-x_{m},x_{n}-x_{m}\rangle A^{\ast}(A^{\ast})^{-1}\|\\
&\leq\|A^{-1}\|^{2}\|A\langle x_{n}-x_{m},x_{n}-x_{m}\rangle A^{\ast}\|.
	\end{align*}
	Therefore the sequence $\{x_{n}\}_{n\in\mathbb{N}}$ is Cauchy and hence there exists $x\in U$ such that $x_{n}\rightarrow x$ as $n\rightarrow\infty$. Again by \eqref{3.3} we have $\|R(x_{n}-x)\|^{2}\leq\|B\|^{2}\|\langle x_{n}-x,x_{n}-x\rangle\|$.
	
	Thus $\|Rx_{n}-Rx\|\rightarrow0$ as $n\rightarrow\infty$ implies that $Rx=y$. It concludes that the range of $R$ is closed. Next we show that $R$ is injective. Suppose that $x\in\mathcal{H}$ and $Rx=0$. Note that $A\langle x,x\rangle A^{\ast}\leq\langle Rx,Rx\rangle$ then $\langle x,x\rangle=0$ so $x=0$ i.e. $R$ is injective.
	
	For $x\in\mathcal{H}$ and $\{x_{i}\}_{i\in I}\in l^{2}(\mathcal{H})$ we have $$\langle Rx, \{x_{i}\}_{i\in I}\rangle=\langle \{T_{i}x\}_{i\in I}, \{x_{i}\}_{i\in I}\rangle=\sum_{i\in I}\langle T_{i}x, x_{i}\rangle=\sum_{i\in I}\langle x, T_{i}^{\ast}x_{i}\rangle=\langle x, \sum_{i\in I}T_{i}^{\ast}x_{i}\rangle.$$ 
	Then $R^{\ast}(\{x_{i}\}_{i\in I})=\sum_{i\in I}T_{i}^{\ast}x_{i}$. By injectivity of $R$, the operator $R^{\ast}$ has closed range and $\mathcal{H}=range(R^{\ast})$, which completes the proof.
\end{proof}
Now we define $\ast$-frame operator and studies some of its properties.
\begin{definition}
Let $\{T_{i}\}_{i\in I}\subset End_{\mathcal{A}}^{\ast}(\mathcal{H})$ be an $\ast$-operator frame with $\ast$-operator frame transform $R$ and lower and upper bounds $A$ and $B$, respectively. The $\ast$-frame operator $S:\mathcal{H}\to\mathcal{H}$ is defined by $Sx=R^{\ast}Rx=\sum_{i\in I}T_{i}^{\ast}T_{i}x,\;\;\forall x\in\mathcal{H}$.
\end{definition}
\begin{theorem}
	The $\ast$-operator frame $S$ is bounded, positive, self-adjoint, invertible and $\|A^{-1}\|^{-2}\leq\|S\|\leq\|B\|^{2}$.
\end{theorem}
\begin{proof}
	By definition we have, $\forall x, y\in\mathcal{H}$:
	\begin{align*}
	\langle Sx,y\rangle&=\left\langle\sum_{i\in I}T_{i}^{\ast}T_{i}x,y\right\rangle\\
	&=\sum_{i\in I}\langle T_{i}^{\ast}T_{i}x,y\rangle\\
	&=\sum_{i\in I}\langle x,T_{i}^{\ast}T_{i}y\rangle \\
	&=\left\langle x,\sum_{i\in I}T_{i}^{\ast}T_{i}y\right\rangle\\
	&=\langle x,Sy\rangle.
	\end{align*}
	Then $S$ is a selfadjoint.
	
	By Lemma \ref{2.8} and Theorem \ref{2.3}, $S$ is invertible. Clearly $S$ is positive.
	
	By definition of an $\ast$-operator frame we have
	\begin{equation*}
	A\langle x,x\rangle A^{\ast}\leq\sum_{i\in I}\langle T_{i}x,T_{i}x\rangle\leq B\langle x,x\rangle B^{\ast}.
	\end{equation*}
	So
	\begin{equation*}
	A\langle x,x\rangle A^{\ast}\leq\langle Sx,x\rangle\leq B\langle x,x\rangle B^{\ast}.
	\end{equation*}
	This give
	\begin{equation*}
	\|A^{-1}\|^{-2}\|x\|^{2}\leq\|\langle Sx,x\rangle\|\leq\|B\|^{2}\|x\|^{2}, \forall x\in\mathcal{H}.
	\end{equation*}
	If we take supremum on all $x\in\mathcal{H}$, where $\|x\|\leq1$, then $\|A^{-1}\|^{-2}\leq\|S\|\leq\|B\|^{2}$.
\end{proof}
\begin{corollary}
Let $\{T_{i}\}_{i\in I}\subset End_{\mathcal{A}}^{\ast}(\mathcal{H})$ be an $\ast$-operator frame with $\ast$-operator frame transform $R$ and lower and upper bounds $A$ and $B$, respectively. Then $\{T_{i}\}_{i\in I}$ is an operator frame for $\mathcal{H}$ with lower and upper bounds $\|(R^{\ast}R)^{-1}\|^{-1}$ and $\|R\|^{2}$, respectively.
\end{corollary}
\begin{proof}
	By Theorem \ref{2.3}, $R$ is injective and has closed range and by Lemma \ref{2.8}
	$$\|(R^{\ast}R)^{-1}\|^{-1}I_{\mathcal{H}}\leq R^{\ast}R\leq\|R\|^{2}I_{\mathcal{H}}.$$
	So$$\|(R^{\ast}R)^{-1}\|^{-1}\langle x, x\rangle\leq \sum_{i\in I}\langle T_{i}x, T_{i}x\rangle\leq\|R\|^{2}\langle x, x\rangle,\;\;\forall x\in\mathcal{H}.$$
	Then $\{T_{i}\}_{i\in I}$ is an operator frame for $\mathcal{H}$ with lower and upper bounds $\|(R^{\ast}R)^{-1}\|^{-1}$ and $\|R\|^{2}$, respectively.
\end{proof}	
The following theorem is similar to theorem 2.3 in \cite{Ali}.
\begin{theorem} \label{Th3.9}
	Let $\{T_{i}\}_{i\in I}\subset End_{\mathcal{A}}^{\ast}(\mathcal{H})$ be an $\ast$-operator frame for $\mathcal{H}$, with lower and upper bounds $A$ and $B$, respectively and with $\ast$-frame operator $S$. Let $\theta\in End_{\mathcal{A}}^{\ast}(\mathcal{H})$ be injective and has a closed range. Then $\{T_{i}\theta\}_{i\in I}$ is an $\ast$-operator frame for $\mathcal{H}$ with $\ast$-frame operator $\theta^{\ast}S\theta$ with bounds $\|(\theta^{\ast}\theta)^{-1}\|^{-\frac{1}{2}}A$, $\|\theta\|B$.
\end{theorem}
\begin{proof}
	We have 
	\begin{equation}\label{eq11}
		A\langle\theta x, \theta x\rangle A^{\ast}\leq\sum_{i\in I}\langle T_{i}\theta x, T_{i}\theta x\rangle\leq B\langle\theta x, \theta x\rangle B^{\ast}, \forall x\in\mathcal{H}.
	\end{equation}
	Using Lemma \ref{2.8}, we have $\|(\theta^{\ast}\theta)^{-1}\|^{-1}\langle x,x\rangle\leq\langle \theta x,\theta x\rangle$, $\forall x\in\mathcal{H}$. This implies
	\begin{equation}\label{eq22} 
		\|(\theta^{\ast}\theta)^{-1}\|^{-\frac{1}{2}}A\langle x,x\rangle(\|(\theta^{\ast}\theta)^{-1}\|^{-\frac{1}{2}}A)^{\ast}\leq A\langle \theta x,\theta x\rangle A^{\ast}, \forall x\in\mathcal{H}.
	\end{equation}
	And we know that $\langle \theta x,\theta x\rangle\leq\|\theta\|^{2}\langle x,x\rangle$, $\forall x\in\mathcal{H}$. This implies that
	\begin{equation}\label{eq33}
		B\langle \theta x,\theta x\rangle B^{\ast}\leq\|\theta\|B\langle x,x\rangle(\|\theta\|B)^{\ast}, \forall x\in\mathcal{H}.
	\end{equation}
	Using \eqref{eq11}, \eqref{eq22}, \eqref{eq33} we have
\begin{align*}
\|(\theta^{\ast}\theta)^{-1}\|^{-\frac{1}{2}}A\langle x,x\rangle(\|(\theta^{\ast}\theta)^{-1}\|^{-\frac{1}{2}}A)^{\ast}&\leq\sum_{i\in I}\langle T_{i}\theta x, T_{i}\theta x\rangle\\&\leq\|\theta\|B\langle x,x\rangle(\|\theta\|B)^{\ast}, \forall x\in\mathcal{H}.
\end{align*}
	So $\{T_{i}\theta\}_{i\in\mathbb{J}}$ is an $\ast$-operator frame for $\mathcal{H}$.
	
	Moreover for every $x\in\mathcal{H}$, we have
	$$\theta^{\ast}S\theta x=\theta^{\ast}\sum_{i\in I}T_{i}^{\ast}T_{i}\theta x=\sum_{i\in I}\theta^{\ast}T_{i}^{\ast}T_{i}\theta x=\sum_{i\in I}(T_{i}\theta)^{\ast}(T_{i}\theta)x.$$ This completes the proof.	
\end{proof}
\begin{corollary}
Let $\{T_{i}\}_{i\in I}\subset End_{\mathcal{A}}^{\ast}(\mathcal{H})$ be an $\ast$-operator frame for $\mathcal{H}$, with $\ast$-frame operator $S$. Then $\{T_{i}S^{-1}\}_{i\in I}$ is an $\ast$-operator frame for $\mathcal{H}$.
\end{corollary}
\begin{proof}
	Result from theorem \ref{Th3.9} by taking $\theta=S^{-1}$.
\end{proof}	
\begin{corollary}
	Let $\{T_{i}\}_{i\in I}\subset End_{\mathcal{A}}^{\ast}(\mathcal{H})$ be an $\ast$-operator frame for $\mathcal{H}$, with $\ast$-frame operator $S$. Then $\{T_{i}S^{-\frac{1}{2}}\}_{i\in I}$ is a Parseval $\ast$-operator frame for $\mathcal{H}$.
\end{corollary}
\begin{proof}
	Result from theorem \ref{Th3.9} by taking $\theta=S^{-\frac{1}{2}}$.
\end{proof}
The following theorem is a generalization of theorem 2.3 in \cite{Ali}.
\begin{theorem} \label{.}
	Let $\{T_{i}\}_{i\in I}\subset End_{\mathcal{A}}^{\ast}(\mathcal{H})$ be an $\ast$-operator frame for $\mathcal{H}$, with lower and upper bounds $A$ and $B$, respectively. Let $\theta\in End_{\mathcal{A}}^{\ast}(\mathcal{H})$ be surjective. Then $\{\theta T_{i}\}_{i\in I}$ is an $\ast$-operator frame for $\mathcal{H}$ with bounds $A\|(\theta\theta^{\ast})^{-1}\|^{-\frac{1}{2}}$, $B\|\theta\|$.
\end{theorem}
\begin{proof}
	By the definition of $\ast$-operator frame, we have
	\begin{equation}\label{eqq11}
	A\langle x, x\rangle A^{\ast}\leq\sum_{i\in I}\langle T_{i}x, T_{i}x\rangle\leq B\langle x, x\rangle B^{\ast}, \forall x\in\mathcal{H}.
	\end{equation}
	Using Lemma \ref{2.8}, we have
	\begin{equation}\label{eqq111} \|(\theta\theta^{\ast})^{-1}\|^{-1}\langle T_{i}x, T_{i}x\rangle\leq\langle \theta T_{i}x,\theta T_{i}x\rangle\leq\|\theta\|^{2}\langle T_{i}x, T_{i}x\rangle, \forall x\in\mathcal{H}.
	\end{equation}
	 	Using \eqref{eqq11}, \eqref{eqq111}, we have
	 \begin{align*}
	 	\|(\theta\theta^{\ast})^{-1}\|^{-\frac{1}{2}}A\langle x,x\rangle(\|(\theta\theta^{\ast})^{-1}\|^{-\frac{1}{2}}A)^{\ast}&\leq\sum_{i\in I}\langle\theta T_{i}x, \theta T_{i}x\rangle\\&\leq B\|\theta\|\langle x,x\rangle(B\|\theta\|)^{\ast}, \forall x\in\mathcal{H}.
	 \end{align*}
	 So $\{\theta T_{i}\}_{i\in\mathbb{J}}$ is an $\ast$-operator frame for $\mathcal{H}$.
\end{proof}
The following theorem is similar of the theorem $3.5$ in \cite{Ros1}.
\begin{theorem}
	Let $(\mathcal{H},\mathcal{A},\langle.,.\rangle_{\mathcal{A}})$ and $(\mathcal{H},\mathcal{B},\langle.,.\rangle_{\mathcal{B}})$ be two Hilbert $\mathcal{C^{\ast}}$-modules and let $\varphi :\mathcal{A}\longrightarrow \mathcal{B}$ be a $\ast$-homomorphism and $\theta$ be a map on $\mathcal{H}$ such that $\langle \theta x,\theta y\rangle_{\mathcal{B}}=\varphi(\langle x, y\rangle_{\mathcal{A}})$ for all $x,y\in\mathcal{H}$. Also, suppose that $\{T_{i}\}_{i\in I}\subset End_{\mathcal{A}}^{\ast}(\mathcal{H})$ is an $\ast$-operator frame for $(\mathcal{H},\mathcal{A},\langle.,.\rangle_{\mathcal{A}})$ with $\ast$-frame operator $S_{\mathcal{A}} $ and lower and upper $\ast$-operator frame bounds $A$, $B$  respectively. If $\theta$ is surjective and $\theta T_{i}= T_{i}\theta$ for each $i$ in $I$, then $\{T_{i}\}_{i\in I}$ is an $\ast$-operator frame for $(\mathcal{H},\mathcal{B},\langle.,.\rangle_{\mathcal{B}})$ with $\ast$-frame operator $S_{\mathcal{B}} $ and lower and upper $\ast$-operator frame bounds $\varphi(A)$, $\varphi(B)$ respectively, and $\langle S_{\mathcal{B}}\theta x,\theta y\rangle_{\mathcal{B}}=\varphi(\langle S_{\mathcal{A}}x, y\rangle_{\mathcal{A}})$.
\end{theorem}
\begin{proof} Let $y\in\mathcal{H}$ then there exists $x\in\mathcal{H}$ such that $\theta x=y$ ($\theta$ is surjective). By the definition of $\ast$-operator frames we have
	$$A\langle x,x\rangle_{\mathcal{A}} A^{\ast}\leq\sum_{i\in I}\langle T_{i}x, T_{i}x\rangle_{\mathcal{A}}\leq B\langle x,x\rangle_{\mathcal{A}} B^{\ast}.$$
	By lemma \ref{2.7} we have
	$$\varphi(A\langle x,x\rangle_{\mathcal{A}} A^{\ast})\leq\varphi(\sum_{i\in I}\langle T_{i}x, T_{i}x\rangle_{\mathcal{A}})\leq\varphi( B\langle x,x\rangle_{\mathcal{A}} B^{\ast}).$$
	By the definition of $\ast$-homomorphism we have
	$$\varphi(A)\varphi(\langle x,x\rangle_{\mathcal{A}}) \varphi(A^{\ast})\leq\sum_{i\in I}\varphi(\langle T_{i}x, T_{i}x\rangle_{\mathcal{A}})\leq\varphi( B)\varphi(\langle x,x\rangle_{\mathcal{A}}) \varphi(B^{\ast}).$$
	By the relation betwen $\theta$ and $\varphi$ we get
	$$\varphi(A)\langle \theta x,\theta x\rangle_{\mathcal{B}} \varphi(A)^{\ast}\leq\sum_{i\in I}\langle \theta T_{i}x,\theta T_{i}x\rangle_{\mathcal{B}}\leq\varphi( B)\langle\theta x,\theta x\rangle_{\mathcal{B}} \varphi(B)^{\ast}.$$
	By the relation betwen $\theta$ and $T_{i}$ we have
	$$\varphi(A)\langle \theta x,\theta x\rangle_{\mathcal{B}} \varphi(A)^{\ast}\leq\sum_{i\in I}\langle T_{i}\theta x, T_{i}\theta x\rangle_{\mathcal{B}}\leq\varphi( B)\langle\theta x,\theta x\rangle_{\mathcal{B}} \varphi(B)^{\ast}.$$
	
	Then
	$$
	\varphi(A)\langle  y, y\rangle_{\mathcal{B}} (\varphi(A))^{\ast}\leq\sum_{i\in I}\langle T_{i}y, T_{i}y\rangle_{\mathcal{B}}
	\leq\varphi( B)\langle y,y\rangle_{\mathcal{B}} (\varphi(B))^{\ast} , \forall y\in\mathcal{H}.
	$$
	On the other hand we have
	$$
	\aligned
	\varphi(\langle S_{\mathcal{A}}x, y\rangle_{\mathcal{A}})&=\varphi(\langle\sum_{i\in I}T_{i}^{\ast}T_{i}x,y\rangle_{\mathcal{A}})\\
	&=\sum_{i\in I}\varphi(\langle T_{i}x, T_{i}y\rangle_{\mathcal{A}})\\
	&=\sum_{i\in I}\langle\theta T_{i}x,\theta T_{i}y\rangle_{\mathcal{B}}
	\\
	&=\sum_{i\in I}\langle T_{i}\theta x, T_{i}\theta y\rangle_{\mathcal{B}}\\
	&=\langle\sum_{i\in I}T_{i}^{\ast}T_{i}\theta x,\theta y\rangle_{\mathcal{B}}\\
	&=\langle S_{\mathcal{B}}\theta x,\theta y\rangle_{\mathcal{B}}.
	\endaligned
	$$
	Which completes the proof.
\end{proof}
\section{Tensor Product}
 Suppose that $\mathcal{A}, \mathcal{B}$ are unital $C^{\ast}$-algebras and  $\mathcal{A}\otimes\mathcal{B}$ is the completion of $\mathcal{A}\otimes_{alg}\mathcal{B}$ with the spatial norm. $\mathcal{A}\otimes\mathcal{B}$ is the spatial tensor product of $\mathcal{A}$ and $\mathcal{B}$, also suppose that $\mathcal{H}$ is a Hilbert $\mathcal{A}$-module and $\mathcal{K}$ is a Hilbert $\mathcal{B}$-module. We want to define $\mathcal{H}\otimes\mathcal{K}$ as a Hilbert $(\mathcal{A}\otimes\mathcal{B})$-module. Start by forming the algebraic tensor product $\mathcal{H}\otimes_{alg}\mathcal{K}$ of the vector spaces $\mathcal{H}$, $\mathcal{K}$ (over $\mathbb{C}$). This is a left module over $(\mathcal{A}\otimes_{alg}\mathcal{B})$ (the module action being given by $(a\otimes b)(x\otimes y)=ax\otimes by$ $(a\in\mathcal{A},b\in\mathcal{B},x\in\mathcal{H},y\in\mathcal{K})$). For $(x_{1},x_{2}\in\mathcal{H},y_{1},y_{2}\in\mathcal{K})$ we define $ \langle x_{1}\otimes y_{1},x_{2}\otimes y_{2}\rangle_{\mathcal{A}\otimes\mathcal{B}}=\langle x_{1},x_{2}\rangle_{\mathcal{A}}\otimes\langle y_{1},y_{2}\rangle_{\mathcal{B}} $. We also know that for $ z=\sum_{i=1}^{n}x_{i}\otimes y_{i} $ in $\mathcal{H}\otimes_{alg}\mathcal{K}$ we have $ \langle z,z\rangle_{\mathcal{A}\otimes\mathcal{B}}=\sum_{i,j}\langle x_{i},x_{j}\rangle_{\mathcal{A}}\otimes\langle y_{i},y_{j}\rangle_{\mathcal{B}}\geq0 $ and $ \langle z,z\rangle_{\mathcal{A}\otimes\mathcal{B}}=0 $ iff $z=0$.
This extends by linearity to an $(\mathcal{A}\otimes_{alg}\mathcal{B})$-valued sesquilinear form on $\mathcal{H}\otimes_{alg}\mathcal{K}$, which makes $\mathcal{H}\otimes_{alg}\mathcal{K}$ into a semi-inner-product module over the pre-$\mathcal{C}^{\ast}$-algebra $(\mathcal{A}\otimes_{alg}\mathcal{B})$.
The semi-inner-product on $\mathcal{H}\otimes_{alg}\mathcal{K}$ is actually an inner product, see \cite{Lan}. Then $\mathcal{H}\otimes_{alg}\mathcal{K}$ is an inner-product module over the pre-$\mathcal{C}^{\ast}$-algebra $(\mathcal{A}\otimes_{alg}\mathcal{B})$, and we can perform the double completion discussed in chapter 1 of \cite{Lan} to conclude that the completion $\mathcal{H}\otimes\mathcal{K}$ of $\mathcal{H}\otimes_{alg}\mathcal{K}$ is a Hilbert $(\mathcal{A}\otimes\mathcal{B})$-module. We call $ \mathcal{H}\otimes\mathcal{K} $ the exterior tensor product of $\mathcal{H}$ and $\mathcal{K}$. With $\mathcal{H}$ , $\mathcal{K}$ as above, we wish to investigate the adjointable operators on $ \mathcal{H}\otimes\mathcal{K} $. Suppose that $S\in End_{\mathcal{A}}^{\ast}(\mathcal{H})$ and $T\in End_{\mathcal{B}}^{\ast}(\mathcal{K})$. We define a linear operator $S\otimes T$ on $ \mathcal{H}\otimes\mathcal{K} $ by $S\otimes T(x\otimes y)=Sx\otimes Ty  (x\in\mathcal{H} ,y\in\mathcal{K})$. It is a routine verification that  is $S^{\ast}\otimes T^{\ast}$ is the adjoint of $S\otimes T$, so in fact $S\otimes T\in End_{\mathcal{A\otimes B}}^{\ast}(\mathcal{H}\otimes\mathcal{K})$. For more details see \cite{Dav,Lan}. We note that if $a\in\mathcal{A}^{+}$ and $b\in\mathcal{B}^{+}$, then $a\otimes b\in(\mathcal{A}\otimes\mathcal{B})^{+}$. Plainly if $a$, $b$ are Hermitian elements of $\mathcal{A}$ and $a\geq b$, then for every positive element $x$ of $\mathcal{B}$, we have $a\otimes x\geq b\otimes x$.

Let I and J be countable index sets. Our next theorem is a generalization of theorem 2.2 in \cite{Ali}.
\begin{theorem}
	Let $\mathcal{H}$ and $\mathcal{K}$ be two Hilbert $C^{\ast}$-modules over unitary $C^{\ast}$-algebras $\mathcal{A}$ and $\mathcal{B}$, respectively. Let $\{\Lambda_{i}\}_{i\in I}\subset End_{\mathcal{A}}^{\ast}(\mathcal{H})$ and $\{\Gamma_{j}\}_{j\in J}\subset End_{\mathcal{B}}^{\ast}(\mathcal{K})$ be two $\ast$-operator frames for $\mathcal{H}$ and $\mathcal{K}$ with $\ast$-frame operators $S_{\Lambda}$ and $S_{\Gamma}$ and $\ast$-operator frame bounds $(A,B)$ and $(C,D)$ respectively. Then $\{\Lambda_{i}\otimes\Gamma_{j}\}_{i\in I,j\in J}  $ is an $\ast$-operator frame for Hibert $\mathcal{A}\otimes\mathcal{B}$-module $\mathcal{H}\otimes\mathcal{K}$ with $\ast$-frame operator $ S_{\Lambda}\otimes S_{\Gamma}$ and lower and upper $\ast$-operator frame bounds $A\otimes C$ and $ B\otimes D $, respectively.
\end{theorem}
\begin{proof}
By the definition of $\ast$-operator frames $\{\Lambda_{i}\}_{i\in I} $ and $\{\Gamma_{j}\}_{j\in J}$ we have 
$$A\langle x,x\rangle_{\mathcal{A}} A^{\ast}\leq\sum_{i\in I}\langle \Lambda_{i}x,\Lambda_{i}x\rangle_{\mathcal{A}}\leq B\langle x,x\rangle_{\mathcal{A}} B^{\ast} , \forall x\in\mathcal{H},$$
and
$$C\langle y,y\rangle_{\mathcal{B}} C^{\ast}\leq\sum_{j\in J}\langle \Gamma_{j}y,\Gamma_{j}y\rangle_{\mathcal{B}}\leq D\langle y,y\rangle_{\mathcal{B}} D^{\ast} , \forall y\in\mathcal{K}.$$
Therefore
$$
\aligned
&(A\langle x,x\rangle_{\mathcal{A}} A^{\ast})\otimes (C\langle y,y\rangle_{\mathcal{B}} C^{\ast})\\&\leq\sum_{i\in I}\langle \Lambda_{i}x,\Lambda_{i}x\rangle_{\mathcal{A}}\otimes\sum_{j\in J}\langle \Gamma_{j}y,\Gamma_{j}y\rangle_{\mathcal{B}}\\
&\leq (B\langle x,x\rangle_{\mathcal{A}} B^{\ast})\otimes (D\langle y,y\rangle_{\mathcal{B}} D^{\ast}) , \forall x\in\mathcal{H} ,\forall y\in\mathcal{K}.
\endaligned
$$
Then
$$
\aligned
&(A\otimes C)(\langle x,x\rangle_{\mathcal{A}}\otimes\langle y,y\rangle_{\mathcal{B}}) (A^{\ast}\otimes C^{\ast})\\&\leq\sum_{i\in I,j\in J}\langle \Lambda_{i}x,\Lambda_{i}x\rangle_{\mathcal{A}}\otimes\langle \Gamma_{j}y,\Gamma_{j}y\rangle_{\mathcal{B}}\\
&\leq (B\otimes D)(\langle x,x\rangle_{\mathcal{A}}\otimes\langle y,y\rangle_{\mathcal{B}}) (B^{\ast}\otimes D^{\ast}) , \forall x\in\mathcal{H} ,\forall y\in\mathcal{K}.
\endaligned
$$
Consequently we have
$$
\aligned
&(A\otimes C)\langle x\otimes y,x\otimes y\rangle_{\mathcal{A\otimes B}} (A\otimes C)^{\ast}\\&\leq\sum_{i\in I,j\in J}\langle\Lambda_{i}x\otimes\Gamma_{j}y,\Lambda_{i}x\otimes\Gamma_{j}y\rangle_{\mathcal{A\otimes B}} \\
&\leq (B\otimes D)\langle x\otimes y,x\otimes y\rangle_{\mathcal{A\otimes B}} (B\otimes D)^{\ast}, \forall x\in\mathcal{H}, \forall y\in\mathcal{K}.
\endaligned
$$
Then for all $x\otimes y\in\mathcal{H\otimes K}$ we have
$$
\aligned
&(A\otimes C)\langle x\otimes y,x\otimes y\rangle_{\mathcal{A\otimes B}} (A\otimes C)^{\ast}\\&\leq\sum_{i\in I,j\in J}\langle(\Lambda_{i}\otimes\Gamma_{j})(x\otimes y),(\Lambda_{i}\otimes\Gamma_{j})(x\otimes y)\rangle_{\mathcal{A\otimes B}} \\
&\leq (B\otimes D)\langle x\otimes y,x\otimes y\rangle_{\mathcal{A\otimes B}} (B\otimes D)^{\ast}.
\endaligned
$$
The last inequality is satisfied for every finite sum of elements in $\mathcal{H}\otimes_{alg}\mathcal{K}$ and then it's satisfied for all $z\in\mathcal{H\otimes K}$. It shows that $\{\Lambda_{i}\otimes\Gamma_{j}\}_{i\in I,j\in J}  $ is $\ast$-operator frame for Hibert $\mathcal{A}\otimes\mathcal{B}$-module $\mathcal{H}\otimes\mathcal{K}$ with lower and upper $\ast$-operator frame bounds $A\otimes C$ and $ B\otimes D $, respectively.\\
By the definition of $\ast$-frame operator $S_{\Lambda}$ and $S_{\Gamma}$ we have:$$S_{\Lambda}x=\sum_{i\in I}\Lambda_{i}^{\ast}\Lambda_{i}x, \forall x\in\mathcal{H},$$
and
$$S_{\Gamma}y=\sum_{j\in J}\Gamma_{j}^{\ast}\Gamma_{j}y, \forall y\in\mathcal{K}.$$
Therefore
$$
\aligned
(S_{\Lambda}\otimes S_{\Gamma})(x\otimes y)&=S_{\Lambda}x\otimes S_{\Gamma}y\\
&=\sum_{i\in I}\Lambda_{i}^{\ast}\Lambda_{i}x\otimes\sum_{j\in J}\Gamma_{j}^{\ast}\Gamma_{j}y\\
&=\sum_{i\in I,j\in J}\Lambda_{i}^{\ast}\Lambda_{i}x\otimes\Gamma_{j}^{\ast}\Gamma_{j}y\\
&=\sum_{i\in I,j\in J}(\Lambda_{i}^{\ast}\otimes\Gamma_{j}^{\ast})(\Lambda_{i}x\otimes\Gamma_{j}y)\\
&=\sum_{i\in I,j\in J}(\Lambda_{i}^{\ast}\otimes\Gamma_{j}^{\ast})(\Lambda_{i}\otimes\Gamma_{j})(x\otimes y)\\
&=\sum_{i\in I,j\in J}(\Lambda_{i}\otimes\Gamma_{j})^{\ast})(\Lambda_{i}\otimes\Gamma_{j})(x\otimes y).
\endaligned
$$
Now by the uniqueness of $\ast$-frame operator, the last expression is equal to $S_{\Lambda\otimes\Gamma}(x\otimes y)$. Consequently we have $ (S_{\Lambda}\otimes S_{\Gamma})(x\otimes y)=S_{\Lambda\otimes\Gamma}(x\otimes y)$. The last equality is satisfied for every finite sum of elements in $\mathcal{H}\otimes_{alg}\mathcal{K}$ and then it's satisfied for all $z\in\mathcal{H\otimes K}$. It shows that $ (S_{\Lambda}\otimes S_{\Gamma})(z)=S_{\Lambda\otimes\Gamma}(z)$. So $S_{\Lambda\otimes\Gamma}=S_{\Lambda}\otimes S_{\Gamma}$.
\end{proof}
The two following theorems are a generalization of Theorem 3.5 in \cite{BA}.
\begin{theorem} \label{Th4.2}
	If $Q\in End_{\mathcal{A}}^{\ast}(\mathcal{H})$ is invertible and $\{\Lambda_{i}\}_{i\in I}\subset End_{\mathcal{A\otimes B}}^{\ast}(\mathcal{H\otimes K})$ is an $\ast$-operator frame for $\mathcal{H\otimes K}$ with lower and upper  $\ast$-operator frame bounds $A$ and $B$ respectively and $\ast$-frame operator $S$, then $\{\Lambda_{i}(Q^{\ast}\otimes I)\}_{i\in I} $ is an $\ast$-operator frame for $\mathcal{H\otimes K}$ with lower and upper $\ast$-operator frame bounds $\|Q^{\ast-1}\|^{-1}A$ and $\|Q\|B$ respectively and $\ast$-frame operator $(Q\otimes I)S(Q^{\ast}\otimes I)$.
\end{theorem}
\begin{proof}
	Since $Q\in End_{\mathcal{A}}^{\ast}(\mathcal{H})$, $Q\otimes I\in End_{\mathcal{A\otimes B}}^{\ast}(\mathcal{H\otimes K})$ with inverse $Q^{-1}\otimes I$. It is obvious that the adjoint of $Q\otimes I$ is $Q^{\ast}\otimes I$. An easy calculation shows that for every elementary tensor $x\otimes y$,
	$$
	\aligned
	\|(Q\otimes I)(x\otimes y)\|^{2}&=\|Q(x)\otimes y\|^{2}\\
	&=\|Q(x)\|^{2}\|y\|^{2}\\
	&\leq\|Q\|^{2}\|x\|^{2}\|y\|^{2}\\
	&=\|Q\|^{2}\|x\otimes y\|^{2}.
	\endaligned
	$$
	So $Q\otimes I$ is bounded, and therefore it can be extended to $\mathcal{H\otimes K}$. Similarly for $Q^{\ast}\otimes I$, hence $Q\otimes I$ is $\mathcal{A\otimes B}$-linear, adjointable with adjoint $Q^{\ast}\otimes I$. Hence for every $z\in\mathcal{H\otimes K}$ we have  $$\|Q^{\ast-1}\|^{-1}.|z|\leq|(Q^{\ast}\otimes I)z|\leq\|Q\|.|z|.$$
	By the definition of $\ast$-operator frames we have 
	$$A\langle z,z\rangle_{\mathcal{A\otimes B}} A^{\ast}\leq\sum_{i\in I}\langle \Lambda_{i}z,\Lambda_{i}z\rangle_{\mathcal{A\otimes B}}\leq B\langle z,z\rangle_{\mathcal{A\otimes B}} B^{\ast}.$$
	Then 
	$$
	\aligned
	&A\langle (Q^{\ast}\otimes I)z,(Q^{\ast}\otimes I)z\rangle_{\mathcal{A\otimes B}} A^{\ast}\\&\leq\sum_{i\in I}\langle \Lambda_{i}(Q^{\ast}\otimes I)z,\Lambda_{i}(Q^{\ast}\otimes I)z\rangle_{\mathcal{A\otimes B}}\\
	&\leq B\langle (Q^{\ast}\otimes I)z,(Q^{\ast}\otimes I)z\rangle_{\mathcal{A\otimes B}} B^{\ast}.
	\endaligned
	$$
	So$$\aligned&\|Q^{\ast-1}\|^{-1}A\langle z,z\rangle_{\mathcal{A\otimes B}}(\|Q^{\ast-1}\|^{-1}A)^{\ast}\\&\leq\sum_{i\in I}\langle \Lambda_{i}(Q^{\ast}\otimes I)z,\Lambda_{i}(Q^{\ast}\otimes I)z\rangle_{\mathcal{A\otimes B}}\\&\leq\|Q\|B\langle z,z\rangle_{\mathcal{A\otimes B}}(\|Q\|B)^{\ast}.\endaligned$$
	Now
	$$
	\aligned
	(Q\otimes I)S(Q^{\ast}\otimes I)&=(Q\otimes I)(\sum_{i\in I}\Lambda_{i}^{\ast}\Lambda_{i})(Q^{\ast}\otimes I)\\
	&=\sum_{i\in I}(Q\otimes I)\Lambda_{i}^{\ast}\Lambda_{i}(Q^{\ast}\otimes I)\\
	&=\sum_{i\in I}(\Lambda_{i}(Q^{\ast}\otimes I))^{\ast}\Lambda_{i}(Q^{\ast}\otimes I).
	\endaligned
	$$
	Which completes the proof.
\end{proof}
\begin{corollary}
	If $Q\in End_{\mathcal{B}}^{\ast}(\mathcal{K})$ is invertible and $\{\Lambda_{i}\}_{i\in I}\subset End_{\mathcal{A\otimes B}}^{\ast}(\mathcal{H\otimes K})$ is an $\ast$-operator frame for $\mathcal{H\otimes K}$ with lower and upper  $\ast$-operator frame bounds $A$ and $B$ respectively and $\ast$-frame operator $S$, then $\{\Lambda_{i}(I\otimes Q^{\ast})\}_{i\in I} $ is an $\ast$-operator frame for $\mathcal{H\otimes K}$ with lower and upper $\ast$-operator frame bounds $\|Q^{\ast-1}\|^{-1}A$ and $\|Q\|B$ respectively and $\ast$-frame operator $(I\otimes Q)S(I\otimes Q^{\ast})$.
\end{corollary}
\begin{proof}
Similar to the proof of the theorem \ref{Th4.2}.
\end{proof}
\bibliographystyle{amsplain}

\end{document}